\documentclass{amsart}
\usepackage{amsmath,amsthm,amssymb}
\usepackage{color}
\usepackage{mathtools}
\mathtoolsset{showonlyrefs=true} 

\makeatletter
    
    \@addtoreset{equation}{section}
  \makeatother

\newtheorem{definition}{Definition}[section]
\newtheorem{proposition}[definition]{Proposition}
\newtheorem{theorem}[definition]{Theorem}
\newtheorem{lemma}[definition]{Lemma}

\newtheorem{remark}{Remark}[section]

\newcommand{\re}{\mathbb R} 

\def \pt{\partial}

\DeclareMathOperator{\diver}{div}

\pagebreak

\title[3D axisymmetric Navier-Stokes equations in periodic slab ]{
Asymptotic behavior and Liouville-type theorems for
axisymmetric stationary Navier-Stokes equations
outside of an infinite cylinder with a periodic boundary condition
}
\author{Hideo Kozono, Yutaka Terasawa and Yuta Wakasugi}

\address[H. Kozono]{Department of Mathematics, Faculty of Science and Engineering,
Waseda University, Tokyo 169--8555, Japan, 
Research Alliance Center of Mathematical Sciences, Tohoku University, 
Sendai 980-8578, Japan}
\email[H. Kozono]{kozono@waseda.jp, hideokozono@tohoku.ac.jp}
\address[Y. Terasawa]{Graduate School of Mathematics, Nagoya University,
Furocho Chikusaku Nagoya 464-8602, Japan}
\email[Y. Terasawa]{yutaka@math.nagoya-u.ac.jp}
\address[Y. Wakasugi]{Graduate School of Advanced Science and Engineering,
Hiroshima University,
Higashi-Hiroshima, 739-8527, Japan}
\email[Y. Wakasugi]{wakasugi@hiroshima-u.ac.jp}
\begin{document}
\begin{abstract}
We study the asymptotic behavior of solutions to
the steady Navier-Stokes equations
outside of an infinite cylinder in
$\re^3$.
We assume that the flow is periodic
in $x_3$-direction and has no swirl.
This problem is closely related with
two-dimensional exterior problem.
Under a condition on the generalized finite Dirichlet integral,
we give a pointwise decay estimate of the vorticity
at the spatial infinity.
Moreover, we prove a Liouville-type theorem
only from the condition of the generalized finite Dirichlet integral.
\end{abstract}

\maketitle
\section{Introduction}
\footnote[0]{2010 Mathematics Subject Classification. 35Q30; 35B53; 76D05}

Let
$D = \{ (x_1, x_2) \in \re^2 ;\, \sqrt{x_1^2+x_2^2} > r_0 \}$
with some constant $r_0>0$,
and let
$S^1 = [-\pi,\pi]$.
We consider
the stationary Navier-Stokes equations
\begin{align}
\label{sns}
	\left\{ \begin{array}{l}
		(v \cdot \nabla )v + \nabla p = \Delta v, \\
		\nabla \cdot v = 0,\\
		v(x_1, x_2, x_3+2\pi) = v(x_1,x_2,x_3),
		\end{array} \right.
		\quad x \in D \times S^1,
\end{align}
where
$v = v(x) = (v_1(x), v_2(x), v_3(x))$
and
$p = p(x)$
denote the velocity vector field and the scalar pressure
at the point $x = (x_1, x_2, x_3)$, respectively.
The last condition in \eqref{sns}
means that the flow is periodic in
$x_3$-direction.

The problem \eqref{sns} is closely related with
the two-dimensional exterior problem
\begin{align}
\label{sns_2d}
	\left\{ \begin{array}{l}
		(v \cdot \nabla )v + \nabla p = \Delta v, \\
		\nabla \cdot v = 0,
		\end{array} \right.
		\quad x=(x_1,x_2) \in D, \quad
		\text{where}\ 
		v = (v_1, v_2).
\end{align}
For \eqref{sns_2d},
Gilbarg and Weinberger \cite{GiWe78}
and the subsequent studies by
Amick \cite{Am88}, Korobkov, Pileckas, and Russo \cite{KoPiRu19, KoPiRu20}
investigated the asymptotic behavior of the solution
at the spatial infinity under the assumption of
finite Dirichlet integral
\begin{align}\label{eq:Diri:2d}
    \int_{D} |\nabla v(x)|^2 \,dx < \infty.
\end{align}
More precisely, they proved that
there exists a constant vector
$v_{\infty} \in \re^2$ such that
\begin{align}
    \lim_{r \to \infty} \sup_{\theta \in (0,2\pi)} | v(r,\theta) - v_{\infty}| = 0,
\end{align}
where
$r = \sqrt{x_1^2+x_2^2}$
and $\theta = \tan^{-1}(x_2/x_1)$.
Moreover, the asymptotic behavior
\begin{align}
    \omega (r,\theta) = o(r^{-3/4}), \quad
    \nabla v (r,\theta) = o(r^{-3/4} (\log r) ) 
\end{align}
uniformly in $\theta \in (0,2\pi)$
as $r \to \infty$
was obtained.
Furthermore, the following Liouville type theorem was also proved:
let
$(v,p)$ be a smooth solution to \eqref{sns_2d} in
$\re^2$ with the finite Dirichlet integral
$\nabla v \in L^2(\re^2)$.
Then, $v$ and $p$ are constant.

Recently,
the authors \cite{KoTeWa_IUMJ} studied \eqref{sns_2d}
under the condition of
generalized finite Dirichlet integral
\begin{align}\label{gene:Diri:2d}
    \int_{D} |\nabla v(x) |^q \,dx < \infty
\end{align}
with some 
$q \in (2,\infty)$.
From the viewpoint of the decay of
$\nabla v$ at the spatial infinity,
the condition \eqref{gene:Diri:2d}
is weaker than \eqref{eq:Diri:2d}.
In \cite{KoTeWa_IUMJ}, the following asymptotic behavior
was obtained:
\begin{align}
    \omega(r, \theta)
    =
    o\left(r^{-\left(1 / q+1 / q^{2}\right)}\right),
    \quad
    \nabla v(r, \theta)
    =
    o\left(r^{-\frac{2}{q}-\frac{1}{q^{2}}+\frac{1}{2}}\right)
\end{align}
uniformly in $\theta \in (0, 2\pi)$ as $r \to \infty$.
Moreover, when $D = \re^2$,
the Liouville type theorem was proved.

It is a natural question to ask
whether similar properties as above are valid in three-dimensional domains.
For this purpose, we consider
the axially symmetric solution.
Let us use the cylindrical coordinates
$r = \sqrt{x_1^2 + x_2^2}$,
$\theta = \tan^{-1} (x_2/x_1)$,
$z = x_3$,
and let
$e_r = (x_1/r, x_2/r, 0)$,
$e_{\theta} = (-x_2/r, x_1/r, 0)$,
and
$e_z = (0,0,1)$.
Using such $\{e_r, e_{\theta}, e_z\}$ as an orthogonal basis in $\re^3$, 
we express the vector field $v = v(x)$ as
\begin{align}
    v(x) = v^r(r,\theta,z) e_r + v^{\theta}(r,\theta, z) e_{\theta} + v^z(r,\theta,z) e_z.
\end{align}
A vector field $v$ and a function
$p=p(r,\theta,z)$
are called axially symmetric 
if they are independent of
$\theta$.
The Navier-Stokes equations
in the axially symmetric case is written as follows:
\begin{align}\label{sns:axi}
    \left\{\begin{aligned}
    &\left(v^{r} \partial_{r}+v^{z} \partial_{z}\right) v^{r}
    -\frac{\left(v^{\theta}\right)^{2}}{r}+\partial_{r} p
    =
    \left(\partial_{r}^{2}+\frac{1}{r} \partial_{r}+\partial_{z}^{2}-\frac{1}{r^{2}}\right) v^{r}, \\
    &\left(v^{r} \partial_{r}+v^{z} \partial_{z}\right) v^{\theta}
    +\frac{v^{r} v^{\theta}}{r}
    =
    \left(\partial_{r}^{2}+\frac{1}{r} \partial_{r}+\partial_{z}^{2}-\frac{1}{r^{2}}\right) v^{\theta}, \\
    &\left(v^{r} \partial_{r}+v^{z} \partial_{z}\right) v^{z}
    +\partial_{z} p
    =
    \left(\partial_{r}^{2}+\frac{1}{r} \partial_{r}+\partial_{z}^{2}\right) v^{z}, \\
    &\partial_{r} v^{r}+\frac{v^{r}}{r}+\partial_{z} v^{z}
    =
    0 .
    \end{aligned}\right.
\end{align}
Moreover, the vorticity
$\omega = \nabla \times v$
is expressed by
$\omega = \omega^r e_r + \omega^{\theta} e_{\theta} + \omega^{z} e_z$,
where
\begin{equation}\label{eqn:1.4}%
	\omega^{r} = - \partial_z v^{\theta},\quad
	\omega^{\theta} = \partial_z v^r - \partial_r v^z,\quad 
	\omega^z = \frac{1}{r} \partial_r (r v^{\theta}).
\end{equation}%
For the axially symmetric Navier-Stokes equations
\eqref{sns:axi} in an exterior
$\mathcal{D}$
of an axisymmetric body in $\re^3$,
under the assumptions
$\nabla v \in L^2(\mathcal{D})$,
$v = 0$ on $\partial \mathcal{D}$,
and
$\lim_{|x|\to 0}v(x)=0$,
Choe and Jin \cite{ChJi09} obtained
the asymptotic behavior of the solution
\begin{align}
    |v^r| + |v^z|
    &= O \left( \left(\frac{\log r}{r} \right)^{\frac{1}{2}} \right),
    \quad
    |v^{\theta}|=
    O \left( \frac{(\log r)^{\frac{1}{8}}}{r^{\frac{3}{8}}} \right),\\
    |\omega^{\theta}|
    &=
    O \left( r^{-\frac{7}{8}} \right)
\end{align}
uniformly in $z$ as $r \to \infty$.
Later on, in the whole space case
$\re^3$,
Weng \cite{We18}
and
Carrillo, Pan, and Zhang \cite{CaPaZh20}
improved the above behavior to
\begin{align}
    |v|
    &= O \left( \left(\frac{\log r}{r} \right)^{\frac{1}{2}} \right),\\
    \left|\nabla v^{r}\right|
    + \left|\nabla v^{z}\right|
    &=
    O \left(
    r^{-\left(\frac{9}{8}\right)^{-}} \right),
    \quad
    \left|\nabla v^{\theta}\right| 
    =
    O \left( r^{-\left(\frac{67}{64}\right)^{-}} \right),\\
    \left|\omega^{r}\right|
    + \left|\omega^{z}\right|
    &=
    O \left( \frac{(\ln r)^{\frac{11}{8}}}{r^{\frac{9} {8}}} \right),
    \quad
\label{eq:CPZ}
    \left|\omega^{\theta}\right|
    =
    O \left( \frac{(\ln r)^{\frac{3}{4}}}{r^{\frac{5}{4}}} \right)
\end{align}
uniformly in $z$ as $r \to \infty$,
where
$a^-$
denotes arbitrary constant less than $a$.
Recently,
Li and Pan \cite{LiPa20} studied the case of
generalized finite Dirichlet integral
$\nabla v \in L^q(\re^3)$
with $q \in (2,\infty)$,
and obtained the asymptotic behavior
\begin{align}\label{eq:LP}
    \left|\omega^{\theta}\right|
    =
    O \left( r^{-\left(\frac{1}{q}+\frac{3}{q^{2}}\right)^-} \right),
    \quad
    \left| \omega^{r}\right|
    +
    \left|\omega^{z}\right|
    =
    O\left(r^{-\left(\frac{1}{q}+\frac{1}{q^{2}}+\frac{3}{q^{3}}\right)^-}\right)
\end{align}
for $q \in [3,\infty)$, provided that
$\sup |v(r_*,z)| \le C$ holds for some
$r_* > 0$;
\begin{align}
    \left|\omega^{\theta}(r, z)\right|=O\left(r^{-(\frac{2}{q}})^-\right), \quad\left|\omega^{r}(r, z)\right|+\left|\omega^{z}(r, z)\right|=O\left(r^{-\left(\frac{1}{q}+\frac{2}{q^{2}}\right)^-}\right)
\end{align}
for $q \in (2,3)$,
provided that
$v^z \to 0$ as $r \to \infty$.

The Liouville type theorem for \eqref{sns:axi}
in
$\re^3$
is still an open question,
but partial results were given by
Lei and Zhang \cite{LeZh11}
with the additional conditions that
$r v^{\theta}$
is bounded and
the stream function is a BMO function.
Later on,
Lei, Zhang, and Zhao
\cite{LeZhZh17}
showed that if
$r v^{\theta} \in L^p(\re^3)$
with some $1 \le p < \infty$,
or if
$\lim_{r\to \infty} r v^{\theta} = 0$,
then $v$ must be a constant.
On the other hand,
Zhao \cite{Zh19}
proved that if
$\nabla v \in L^2(\re^3)$,
$\lim_{|x| \to \infty} v(x) = 0$,
and if 
$|v| \le C(1+r)^{-(2/3)^+}$
or 
$|\omega| \le C (1+r)^{-(5/3)^+}$
holds,
then $v$ must be zero.

Recently,
Carrillo, Pan, Zhang and Zhao \cite{CaPaZhZh20}
studied the Liouville type theorem for
axially symmetric problem
in a periodic slab domain
\begin{align}
    \left\{\begin{aligned}
    &(v \cdot \nabla) v+\nabla p = \Delta v \quad  \text { in } \quad \mathbb{R}^{2} \times S^{1}, \\
    &\nabla \cdot v = 0, \\
    &v \left(x_{1}, x_{2}, x_3\right)=v\left(x_{1}, x_{2}, x_3+2 \pi\right), \\
    &\lim _{|x| \rightarrow \infty} v=0.
\end{aligned}\right.
\end{align}
They proved that,
Under the condition of finite Dirichlet integral
$\nabla v \in L^2(\re^2 \times S^1)$,
$v \equiv 0$.

In this paper, to investigat more detailed properties of the solution,
we further assume that the solution has no swirl,
that is,
$v^{\theta} \equiv 0$.
In this case, the Navier-Stokes equations
are written as
\begin{align}
\label{sns:axi:nosw}
	\left\{ \begin{array}{l}
		\displaystyle (v^r \pt_r + v^z \pt_z ) v^r + \pt_r p
			= \left( \pt_r^2 + \frac{1}{r}\pt_r + \pt_z^2 - \frac{1}{r^2} \right) v^r,\\
		\displaystyle (v^r \pt_r + v^z \pt_z ) v^z + \pt_z p
			= \left( \pt_r^2 + \frac{1}{r}\pt_r + \pt_z^2 \right) v^z,\\
		\displaystyle \pt_r v^r + \frac{v^r}{r} + \pt_z v^z = 0,
	\end{array} \right.
\end{align}
Moreover,
since $v^{\theta}=0$,
the vorticity
$\omega$
satisfies
$\omega^{r} = \omega^{z} = 0$,
and the equation of
$\omega^{\theta}$ becomes
\begin{align}
    (v^r \partial_r + v^z \partial_z) \omega^{\theta}
    - \frac{v^r}{r} \omega^{\theta}
    = \left(
        \partial_r^2 + \frac{1}{r}\partial_r + \partial_z^2 - \frac{1}{r^2}
    \right) \omega^{\theta}.
\end{align}
Introducing a new unknown function
$\Omega = \omega^{\theta}/r$,
we write the equation above to
simpler form
\begin{align}\label{eq:Omega}
    - \left(
        \partial_r^2 + \partial_z^2 + \frac{3}{r} \partial_r
    \right) \Omega
    + \left(
        v^r \partial_r + v^z \partial_z
    \right) \Omega = 0.
\end{align}
Here, we remark that
the fact
$\omega^{\theta}(0,z) = 0$
(cf. Liu and Wang \cite[Lemma 2.2]{LiWa09})
guarantees continuity of
$\Omega(r,z)$
as
$r\to +0$.
This equation has similar properties to those of 
two-dimensional vorticity equation
such as maximum principle.
When the domain is an outside of a cylinder
$D \times \re$,
the authors \cite{KoTeWa_pre}
investigated the asymptotic behavior
of $\Omega$ and $\omega^{\theta}$.
Under the assumptions
\begin{align}
\label{eq:Diri:nonperi}
    \int_{D \times \re} |\nabla v(x)|^q \,dx < \infty
\end{align}
with some
$q \in [2,\infty)$
and
$|v (r,z) | \le C(1+r)^k$
with some
$k \in \re$,
they obtained
\begin{align}\label{eq:KTW}
    |\Omega| &=
    o\left(
        r^{-\left( 1+\frac{3}{q} -\frac{1}{2q} \max\{0,1+k\} \right)}
    \right),\quad
    |\omega^{\theta}|
    =
    o\left(
        r^{-\left(\frac{3}{q} -\frac{1}{2q} \max\{0,1+k\} \right)}
    \right)
\end{align}
uniformly in $z$ as $r \to \infty$.

Concerning the Liouville type theorem
in the swirl free case,
Koch, Nadirashvili, Seregin, and Sverak
\cite{KoNaSeSv09}
proved that
if $v$ is a bounded weak solution to \eqref{sns:axi:nosw} in $\re^3$,
then $v$ is a constant.
Lei, Zhang, and Zhao \cite{LeZhZh17}
proved that
if $\lim_{r \to \infty} |\Omega| = 0$
uniformly in $z$,
then $\omega^{\theta} = 0$ and
$v$ is harmonic.
In particular, if
$v$ is sublinear, that is,
$|v| = o(|x|)$ as $|x| \to \infty$,
then $v$ must be a constant.
Also, Pan and Li \cite{PaLi21NA}
showed that if
$v$ satisfies
\begin{align}
    \left|v^{r}(r, z)\right|+\left|v^{z}(r, z)\right| &\leq C(r+|z|+c)^{\alpha},\\
    \left|v^{r}(r, z)\right|+\left|v^{z}(r, z)-v^{z}(0, z)\right| &\leq C r
\end{align}
with some $\alpha < 1$ and $C>0$,
then $v$ must be a constant.
In \cite{KoTeWa_pre}, the authors
obtained the Liouville type result
under
$\nabla v \in L^q(\re^3)$ with some
$q \in [2,\infty)$
and
\begin{align}
    v^{r}(r, z) \geq-C(1+r)^{k}, \quad(\operatorname{sign} z) v^{z}(r, z) \geq-C(1+r)^{k}
\end{align}
with some
$k \le q + 1$.
This condition does not require
upper bounds of
$v^r$ and $({\rm sign\,}z)v^z$.

In this paper, we study
the equation \eqref{sns},
that is, the case of periodic slab domain.
We put the following condition of
finite generalized Dirichlet integral.
\begin{align}\label{eq:Diri}
     \int_{D \times S^1} |\nabla v(x)|^q \,dx < \infty
\end{align}
with some $q \in [2,\infty)$.
We remark that, formally,
the assumption \eqref{eq:Diri}
is weaker than
\eqref{eq:Diri:nonperi},
since no decay is required for $z$-direction
in \eqref{eq:Diri}.
Our main result on the asymptotic behavior
of $\Omega$ and $\omega^{\theta}$
is as follows:
\begin{theorem}\label{thm_asymp}
Let $(v,p)$ be a smooth axisymmetric solution of \eqref{sns}
with no swirl
satisfying \eqref{eq:Diri}
with some $q \in [2,\infty)$.
Then, we have
\begin{align}
    \begin{dcases}
    \lim_{r\to \infty} r^{ 1+ 2/q - 1/q^2}
	\sup_{z \in [-\pi, \pi]} |\Omega(r,z)| = 0
	&(q>2),\\
	\lim_{r\to \infty} r^{3/2}
	\sup_{z \in [-\pi, \pi]} |\Omega(r,z)| = 0
	&(q=2)
    \end{dcases}
\end{align}
and
\begin{align}
    \begin{dcases}
    \lim_{r\to \infty} r^{2/q - 1/q^2}
	\sup_{z \in [-\pi, \pi]} |\omega^{\theta} (r,z)| = 0
	&(q>2),\\
	\lim_{r\to \infty} r^{1/2}
	\sup_{z \in [-\pi, \pi]} |\omega^{\theta} (r,z)| = 0
	&(q=2).
    \end{dcases}
\end{align}
\end{theorem}
\begin{remark}
Compared to the results in the non periodic cases
\eqref{eq:CPZ}, \eqref{eq:LP}, and \eqref{eq:KTW},
the decay rates in Theorem \ref{thm_asymp}
are in general slightly worse.
This seems to be due to the weaker assumption
\eqref{eq:Diri} and that we do not assume
any assumption on the pointwise behavior of $v$.
\end{remark}

Moreover, by modifying the method of the
proof of Theorem \ref{thm_asymp},
we obtain the Liouville type theorem
in $\re^2 \times S^1$:
\begin{theorem}\label{thm:liouville}
Let
$D = \re^2$
and let
$(v,p)$
be an axisymmetric smooth solution to \eqref{sns}
with no swirl.
We assume \eqref{eq:Diri} with some
$q \in [2,\infty)$.
Then, $v$ must be a constant vector.
\end{theorem}
\begin{remark}
{\rm (i)}
We do not assume the boundedness or
pointwise behavior of
$v$ or $\Omega$,
and hence,
Theorem \ref{thm:liouville}
is not included by those of \cite{CaPaZhZh20, KoNaSeSv09, LeZhZh17, PaLi21NA}.
\par
\noindent
{\rm (ii)}
Applying
Theorem \ref{thm_asymp} in $D = \re^2$,
we have, in particular,
$\lim_{r \to \infty} |\Omega| = 0$
uniformly in $z$.
This enables us to make use of the result by
Lei, Zhang, and Zhao \cite{LeZhZh17}
so that the same conclusion as
Theorem \ref{thm:liouville} may be obtained.
However,
we shall give another independent proof of
Theorem \ref{thm:liouville}
in Section 3.
The advantage of our proof is that the argument is simpler 
and does not rely on the maximum principle.
\par
\noindent
{\rm (iii)}
Recently, Bang,Gui, Wang and Xie \cite{BaGuWaXi} proved the Liouville-type theorem 
in $\re^2\times S^1$
in such a way that any bounded solution $v$ has the form $v=(0, 0, c)$ with some 
constant $c \in \re$ provided either $v^r$ or $v^\theta$ is axisymmetric, or provided 
$\displaystyle{\lim_{r\to\infty}rv^r = 0}$.     
\end{remark}
%
\section{Proof of Theorem \ref{thm_asymp}}
In what follows, we denote by $C$ generic constants.
In particular, $C = C(*,\ldots,*)$
denotes constants depending only on the quantities appearing in the parenthesis.
We sometimes use the symbols
$\diver_{r,z}, \nabla_{r,z}$, and $\Delta_{r,z}$,
which mean the differential operators defined by
$\diver_{r,z} (f_1, f_2) (r,z) = \partial_r f_1(r,z) + \partial_z f_2(r,z)$,
$\nabla_{r,z} f(r,z) = (\partial_r f, \partial_z f)(r,z)$,
and
$\Delta_{r,z} f(r,z) = \partial_r^2 f (r,z) + \partial_z^2 f(r,z)$,
respectively.

Since
\begin{align*}
	|\nabla_x v(x)|^2 = | \nabla_{r,z} v |^2 + \frac{1}{r^2} |v^r(r,z)|^2
\end{align*}
for the axisymmetric vector field $v$ without swirl,
we first note that the condition \eqref{eq:Diri} implies
\begin{align*}
	\infty > \int_{D \times S^1} |\nabla v(x)|^q \,dx
		\ge C \int_{D \times S^1} \left[ |\nabla_{r,z} v (r,z)|^q + r^{-q} |v^r (r,z)|^q \right] r \,drdz,
\end{align*}
and hence, $\omega^{\theta}$ and $\Omega$ satisfy
\begin{align}
\label{Diri_o}
	\int_{-\pi}^{\pi} \int_{r_0}^{\infty} |\omega^{\theta}(r,z)|^q r \,drdz < \infty,\quad
	\int_{-\pi}^{\pi} \int_{r_0}^{\infty}  r^{q+1} | \Omega(r,z) |^q \,drdz < \infty,
\end{align}
respectively.

\subsection{Preliminary estimates}
When $q >2$.
following the argument of Li and Pan \cite{LiPa20},
we prepare a pointwise growth estimate in slab domain.
\begin{lemma}\label{lem:morrey}
Let
$r_0 > 0$
and let
$f = f(r,z) : [r_0, \infty) \times [-\pi,\pi] \to \re$
satisfy
\begin{align}
    \int_{-\pi}^{\pi} \int_{r_0}^{\infty}
    | \nabla_{r,z} f(r,z)|^q r \,drdz < \infty
\end{align}
with some
$q \in (2,\infty)$.
Then, there exists $C>0$ depending only on
$r_0, q$
such that
\begin{align}
    |f(r,z)| \le C (1+r)^{1-2/q}.
\end{align}
\end{lemma}
\begin{proof}
Let
$R > 4 r_0$
be a parameter, and consider
the transformation
$r = R \rho$
and
\begin{align}
    \phi (\rho, z ) = f(R\rho, z).
\end{align}
Since $2 < q < \infty$, it follows from Morrey's inequality 
(see,e.g., Gilbarg-Trudinger \cite[Theorem 7.17]{GiTr}) that for 
$(\rho_1, z_1), (\rho_2, z_2) \in [1/4,4]\times [-\pi,\pi]$,
\begin{align}
    | \phi(\rho_1,z_1) - \phi(\rho_2,z_2)|
    &\le
    C \left(
    \int_{-\pi}^{\pi}\int_{1/4}^4 |\nabla_{\rho,z} \phi(\rho,z)|^q \,d\rho dz \right)^{1/q}
\end{align}
with some constant
$C = C(q) > 0$.
Since $R \ge 4r_0$ and
$\nabla_{\rho,z} \phi(\rho,z) = ( R \partial_r f, \partial_z f)(R\rho, z)$,
we obtain
\begin{align}
    C \left(
    \int_{-\pi}^{\pi}\int_{1/4}^4 |\nabla_{\rho,z} \phi(\rho,z)|^q \,d\rho dz \right)^{1/q}
    &\le
    C R
    \left(
    \int_{-\pi}^{\pi} \int_{1/4}^{4} |\nabla_{r,z} f(R \rho,z)|^q \,d\rho dz
    \right)^{1/q} \\
    &\le
    C R^{1-1/q}
    \left(
        \int_{-\pi}^{\pi} \int_{R/4}^{4R}
        |\nabla_{r,z} f(r,z)|^q \,drdz
    \right)^{1/q}\\
    &\le C R^{1-2/q}
     \left(
        \int_{-\pi}^{\pi} \int_{R/4}^{4R}
        |\nabla_{r,z} f(r,z)|^q r \,drdz
    \right)^{1/q}.
\end{align}
From this,
putting $r_1 = R\rho_1$
and $r_2 = R\rho_2$,
we have,
for any
$(r_1,z_1), (r_2, z_2) \in [-R/4, 4R] \times [-\pi, \pi]$,
\begin{align}
    |f(r_1,z_1) - f(r_2, z_2)|
    \le C R^{1-2/q}
    \left(
    \int_{-\pi}^{\pi} \int_{r_0}^{\infty}
    |\nabla_{r,z} f(r,z)|^q r \,drdz
    \right)^{1/q},
\end{align}
where we note that the constant
$C$
is independent of
$R$.
Now, fix a point
$r_* \in (r_0,\infty)$,
and let
$r > r_*, z \in S^1$.
There exists a nonnegative integer
$n$
such that 
$2^{n} \le r/r_* \le 2^{n+1}$.
Therefore, we have
\begin{align}
    | f(r,z) - f(r_*, z) |
    &\le
    \sum_{j=0}^{n-1}
    \left|
        f \left( \frac{r}{2^j},z\right)
        - f \left( \frac{r}{2^{j+1}}, z \right) 
    \right|
    +
    \left| f \left( \frac{r}{2^n}, z \right)
    - f(r_*,z) \right|
    \\
    &\le
    C \sum_{j=0}^{\infty} \left( \frac{r}{2^j} \right)^{1-2/q} \\
    &\le
    C r^{1-2/q}.
\end{align}
Furthermore, it follows again from Morrey's inequality that 
\begin{align}
    | f(r_*, z) - f(r_*,0) | 
    \le
    C.
\end{align}
Now, by the above two estimates we have  
\begin{align}
    | f(r,z) - f(r_*, 0) |
    \le C (1+r)^{1-2/q}.
\end{align}
This completes the proof.
\end{proof}

For the case
$q =2$,
we recall the following estimate
proved by Gilbarg and Weinberger \cite[Lemma 2.1]{GiWe78}.
\begin{lemma}\label{lem:GiWe}
Let
$f = f(r,z) \in C^1((r_0,\infty) \times S^1)$
satisfy
\begin{align}
    \int_{-\pi}^{\pi} \int_{r_0}^{\infty} |\partial_r f(r,z)|^2 r \,drdz < \infty.
\end{align}
Then, we have
\begin{align}
    \lim_{r\to \infty} \frac{1}{\log r}
    \int_{-\pi}^{\pi} |f(r,z)|^2 \,dz = 0.
\end{align}
\end{lemma}

\subsection{$L^q$-energy estimates}

\begin{lemma}\label{lem_en_est}
Suppose that the assumptions of Theorem \ref{thm_asymp} holds. 
Let $r_1 > r_0$,
and let $\alpha \in \re$ be
\begin{align}
    \alpha = \begin{dcases}
    q + 1 - 2/q & (q>2),\\
     1 & (q=2).
    \end{dcases}
\end{align}
Then, we have
\begin{align*}
	\int_{-\pi}^{\pi} \int_{r_1}^{\infty} r^{\alpha} |\Omega(r,z)|^{q-2} |\nabla \Omega(r,z)|^2 \,drdz
	\le
	C \int_{-\pi}^{\pi} \int_{r_0}^{\infty} r^{q+1} |\Omega(r,z)|^q \,drdz, 
\end{align*}
where $C = C(q, \alpha, r_0, r_1)$.  
\end{lemma}
\begin{proof}
Let
$r_2$ be
$r_1 > r_2 > r_0$, and let
$\xi_1(r), \xi_2(r) \in C^{\infty}((0,\infty))$
be nonnegative functions satisfying
\begin{align}
\label{eq:testfunc}
	\xi_1(r) = \begin{cases} 1 &(r \ge r_1),\\ 0 &(r_0< r \le r_2), \end{cases}\quad
	\xi_2(r) = \begin{cases}
	1 &(0\le r \le 1/2),\\
	\text{decreasing} &(1/2 < r < 1),\\
	0 &(r \ge 1). \end{cases}
\end{align}
For a parameter
$R > \max\{2r_0, 1\}$,
we define a cut-off function
\begin{align*}
	\eta_R(r) = \xi_1(r) \xi_2 \left( \frac{r}{R} \right).
\end{align*}
Then, we see that
\begin{align}\label{est_eta}
    \begin{aligned}
	|\nabla_{r,z} \eta_R(r)| &\le C (|\xi_1'(r)| + R^{-1}),\\
	|\Delta_{r,z} \eta_R(r)| &\le C (|\xi_1''(r)| + R^{-1}|\xi_1'(r)| + R^{-2} ).
	\end{aligned}
\end{align}

Let $h= h(\Omega)$ be a
$C^1$ and piecewise $C^2$ function determined later.
Based on the idea of \cite[p.385]{GiWe78},
we consider the following identity.
\begin{align*}
	&\diver_{r,z} \left[ (r^{\alpha} \eta_R) \nabla_{r,z}h(\Omega)
					- h(\Omega) \nabla_{r,z} (r^{\alpha} \eta_R )
					- (r^{\alpha} \eta_R) h(\Omega) v \right] \\
	&= r^{\alpha} \eta_R h''(\Omega) |\nabla_{r,z} \Omega|^2
		- h(\Omega) \left[ \Delta_{r,z} (r^{\alpha} \eta_R)
						+ v \cdot \nabla_{r,z} (r^{\alpha} \eta_R ) \right] \\
	&\quad + r^{\alpha} \eta_R h'(\Omega) \left[ \Delta_{r,z} \Omega - v \cdot \nabla_{r,z} \Omega \right]
		- r^{\alpha} \eta_R h(\Omega) \diver_{r,z} v.
\end{align*}
A straightforward computation shows
\begin{align}\label{eq:deri:test}
    \begin{aligned}
    \Delta_{r, z}\left(r^{\alpha} \eta_{R}\right)
    &=
    \alpha(\alpha-1) r^{\alpha-2} \eta_{R}
    +2 \alpha r^{\alpha-1} \partial_{r} \eta_{R}
    +r^{\alpha} \partial_r^2 \eta_R,\\
    v \cdot \nabla_{r, z}
    \left(r^{\alpha} \eta_{R}\right)
    &=
    \alpha r^{\alpha-1} v^{r} \eta_{R}
    +r^{\alpha} v^r \partial_r \eta_{R}.
    \end{aligned}
\end{align}
Applying the equation \eqref{eq:Omega},
we calculate
\begin{align}\label{eq:divform:eqOmega}
	& r^{\alpha} \eta_R h'(\Omega) \left[ \Delta_{r,z} \Omega - v \cdot \nabla_{r,z} \Omega \right] \\
	&= r^{\alpha} \eta_R h'(\Omega) \left( - \frac{3}{r} \partial_r \Omega \right) \\
	&= - 3 \partial_r \left( r^{\alpha-1} \eta_R h(\Omega) \right)
		+ 3 \partial_r \left( r^{\alpha-1} \eta_R \right) h(\Omega),\\
	&=
	- 3 \partial_r \left( r^{\alpha-1} \eta_R h(\Omega) \right)
    + 3(\alpha-1) r^{\alpha-2} \eta_{R} h(\Omega)
    + 3 r^{\alpha-1} \partial_{r} \eta_{R} h(\Omega).
\end{align}
Also, from the third line of \eqref{sns:axi:nosw}, we obtain
\begin{align}\label{eq:divform:divv}
    - r^{\alpha} \eta_R h(\Omega) \diver_{r,z} v = r^{\alpha-1} \eta_R h(\Omega) v^r.
\end{align}
Therefore, we see from
\eqref{eq:deri:test},
\eqref{eq:divform:eqOmega},
and
\eqref{eq:divform:divv}
that
\begin{align}
\label{eq:div:form:2}
	&\diver_{r,z} \left[ (r^{\alpha} \eta_R) \nabla_{r,z}h(\Omega)
					- h(\Omega) \nabla_{r,z} (r^{\alpha} \eta_R )
					- (r^{\alpha} \eta_R) h(\Omega) v \right] \\
	&\quad 
	+ 3 \partial_r (r^{\alpha-1} \eta_R h(\Omega)) \\
	&= r^{\alpha} \eta_R h''(\Omega) |\nabla_{r,z} \Omega|^2
		- h(\Omega) \left[ r^{\alpha} \partial_{r}^2 \eta_R + (2\alpha-3) r^{\alpha-1} \partial_r \eta_R \right] \\
	&\quad - h(\Omega) \left[ r^{\alpha} v \partial_{r} \eta_R + (\alpha-1) r^{\alpha-1} v^r \eta_R \right] \\
	&\quad
	-(\alpha-3)(\alpha-1) h(\Omega) r^{\alpha-2} \eta_R.
\end{align}

Now, we divide the proof into two cases for 
$q>2$ and for $q=2$.
First, when
$q > 2$,
we take
\begin{align}
    h(\Omega) = |\Omega|^q. 
\end{align}
Integrating both sides of the above identity (\ref{eq:div:form:2}) over
$D \times S^1$, 
we have by periodicity in $z$-direction that 
\begin{align*}
	&q(q-1) \int_{D\times S^1} r^{\alpha} \eta_R |\Omega(r,z)|^{q-2} |\nabla_{r,z} \Omega(r,z) |^2 \,drdz \\
	&= \int_{D\times S^1} |\Omega(r,z)|^q
		\left[ r^{\alpha} \partial_{r}^2 \eta_R + (2\alpha - 3) r^{\alpha-1} \partial_r \eta_R \right] \, drdz \\
	&\quad + \int_{D\times S^1}  |\Omega(r,z)|^q
		\left[ r^{\alpha} v \partial_{r} \eta_R + (\alpha-1) r^{\alpha-1} v^r \eta_R \right] \,drdz \\
	&\quad + (\alpha - 3)(\alpha-1) \int_{D\times S^1} |\Omega(r,z)|^q r^{\alpha-2} \eta_R \,drdz.
\end{align*}
Applying \eqref{est_eta}
and Lemma \ref{lem:morrey},
we have by the property of the support of $\eta_R$ and its derivatives that  
\begin{align*}%
	&\int_{D\times S^1} r^{\alpha} \eta_R |\Omega(r,z)|^{q-2} |\nabla_{r,z} \Omega(r,z) |^2 \,drdz \\
	&\le
	C \sum_{l=0}^2 R^{-l}
	\int_{-\pi}^{\pi} \int_{\frac{R}{2}}^{R}  r^{\alpha-2 + l} |\Omega(r,z)|^q \,drdz \\
	&\quad
	+ C R^{-1}
	\int_{-\pi}^{\pi} \int_{\frac{R}{2}}^{R} 
	r^{1-2/q+\alpha} |\Omega(r,z)|^q \,drdz \\
	&\quad
	+ C \int_{-\pi}^{\pi} \int_{r_0}^{R} 
	    (r^{-2/q+\alpha}+r^{\alpha-2})
	    |\Omega(r,z)|^q \,drdz \\
	&\quad
	+ C \int_{-\pi}^{\pi} \int_{r_2}^{r_1}
	|\Omega(r,z)|^q \,drdz \\
	&\le C(q, \alpha, r_0, r_1) \int_{D\times S^1}
	r^{-2/q+\alpha} |\Omega(r,z)|^q \,drdz.
\end{align*}%
From the assumption
$\alpha = q+1+2/q$,
we obtain
\begin{align*}%
	&\int_{D\times S^1} r^{\alpha} \eta_R
	|\Omega(r,z)|^{q-2} |\nabla_{r,z} \Omega(r,z) |^2
	\,drdz 
	\le
	C \int_{D\times S^1}
	r^{q+1} |\Omega(r,z)|^q \,drdz.
\end{align*}%
Finally, noting
$\eta_R = 1$
on $[r_1, R/2]$
and letting $R\to \infty$,
we conclude that 
\begin{align*}%
	\int_{D_1\times S^1} r^{\alpha} |\Omega(r,z)|^{q-2} |\nabla_{r,z} \Omega(r,z) |^2 \,drdz
	\le
	C \int_{D\times S^1}
	r^{q+1} |\Omega(r,z)|^q \,drdz.
\end{align*}%

Next, we consider the case
$q=2$.
In this case, we take
$\alpha = 1$
and
\begin{align}
    h(\Omega) = 
    \begin{dcases}
    |\Omega|^2 &(|\Omega| \le A),\\
    A(2|\Omega| - A) &(|\Omega| \ge A).
    \end{dcases}
\end{align}
with a positive constant $A$.
Then, $h(\Omega)$ is $C^1$ and piecewise $C^2$,
and
$|h(\Omega)| \le \min\{ |\Omega|^2, 2A |\Omega| \}$
holds.
In this case, the identity \eqref{eq:div:form:2} has the form
\begin{align}
    &\diver_{r,z} \left[ (r \eta_R) \nabla_{r,z}h(\Omega)
					- h(\Omega) \nabla_{r,z} (r \eta_R )
					- (r \eta_R) h(\Omega) v \right]
					+ 3 \partial_r (\eta_R h(\Omega)) \\
	&= r \eta_R h''(\Omega) |\nabla_{r,z} \Omega|^2
		- h(\Omega) \left[ r \partial_{r}^2 \eta_R
		- \partial_r \eta_R \right] \\
	&\quad - h(\Omega) r v^r \partial_r \eta_R.
\end{align}
Integrating both sides of the identity above over $D\times S^1$,
we have
\begin{align}\label{eq:energy:q=2}
    2 \int_{D\times S^1 \atop |\Omega| \le A}
    r \eta_R |\nabla_{r,z} \Omega|^2
    \,drdz
    &=
     \int_{D\times S^1}
     h(\Omega) \left[ r \partial_{r}^2 \eta_R
		- \partial_r \eta_R \right]\,drdz\\
	&\quad
	+
	\int_{D\times S^1}
	h(\Omega) r v^r \partial_r \eta_R \,drdz.
\end{align}
The first term of the right-hand side
can be estimated in the same way as before.
For the second term,
we put
\begin{align}
    \overline{v^r} =
    \frac{1}{2\pi} \int_{-\pi}^{\pi} v^r(r,z) \,dz.
\end{align}
Then, we have
\begin{align}
    \int_{D\times S^1}
	h(\Omega) r v^r \partial_r \eta_R \,drdz
	&=
	\int_{D\times S^1}
	h(\Omega) r (v^r-\overline{v^r}) \partial_r \eta_R \,drdz \\
	&\quad 
	+ \overline{v^r}
	\int_{D\times S^1}
	h(\Omega) r \partial_r \eta_R \,drdz \\
	&=: I + I\!I.
\end{align}
For the term $I$,
we use
\eqref{est_eta}
and
$|h(\Omega)| \le \min\{ |\Omega|^2, 2A |\Omega|\}$,
and
we apply the Schwarz and the Wirtinger inequality
to obtain
\begin{align}
    |I| &\le
    C A R^{-2} \left( \int_{-\pi}^{\pi} \int_{\frac{R}{2}}^R r^3 |\Omega|^2 \,drdz \right)^{1/2}
    \left(
    \int_{\frac{R}{2}}^R
       r \left(
       \int_{-\pi}^{\pi} |\partial_{z}v^r|^2 \,dz
       \right)
       \,dr
    \right)^{1/2} \\
    &\quad
    + C \int_{-\pi}^{\pi} \int_{r_2}^{r_1}
	|\Omega(r,z)|^2 \,drdz\\
    &\le
    C A R^{-2}
    \left( \int_{-\pi}^{\pi} \int_{\frac{R}{2}}^R r^3 |\Omega|^2 \,drdz \right)^{1/2}
    \| \nabla v \|_{L^2} \\
    &\quad
    + C \int_{-\pi}^{\pi} \int_{r_2}^{r_1}
	|\Omega(r,z)|^2 \,drdz.
\end{align}
For the term
$I\!I$,
we use
\eqref{est_eta} and
$|h(\Omega)| \le |\Omega|^2$,
and we apply Lemma \ref{lem:GiWe}
to deduce
\begin{align}
    |I\!I|
    &\le
    C (\log R)^{1/2} R^{-3}
    \int_{-\pi}^{\pi} \int_{\frac{R}{2}}^{R}
    r^3 |\Omega|^2 \,drdz
    + C \int_{-\pi}^{\pi} \int_{r_2}^{r_1}
	|\Omega(r,z)|^2 \,drdz
\end{align}
Applying these estimates to \eqref{eq:energy:q=2},
we conclude
\begin{align}
    \int_{D\times S^1 \atop |\Omega| \le A}
    r \eta_R |\nabla_{r,z} \Omega|^2
    \,drdz
    &\le
    C A R^{-2}
    \left( \int_{-\pi}^{\pi} \int_{\frac{R}{2}}^R r^3 |\Omega|^2 \,drdz \right)^{1/2}
    \| \nabla v \|_{L^2} \\
    &\quad
    + C (\log R)^{1/2} R^{-3}
    \int_{-\pi}^{\pi} \int_{\frac{R}{2}}^{R}
    r^3 |\Omega|^2 \,drdz \\
    &\quad
    + C \int_{-\pi}^{\pi} \int_{r_2}^{r_1}
	|\Omega(r,z)|^2 \,drdz.
\end{align}
Letting
$R \to \infty$
yields
\begin{align}
     \int_{D_1 \times S^1 \atop |\Omega| \le A}
    r |\nabla_{r,z} \Omega|^2
    \,drdz
    \le
    C \int_{-\pi}^{\pi} \int_{r_2}^{r_1}
	|\Omega(r,z)|^2 \,drdz.
\end{align}
Since the right-hand side is independent of $A$,
we obtain
\begin{align}
    \int_{-\pi}^{\pi} \int_{r_1}^{\infty}
    r |\nabla_{r,z} \Omega|^2
    \,drdz
    \le
    C \int_{-\pi}^{\pi}\int_{r_0}^{\infty} r^3
	|\Omega(r,z)|^2 \,drdz
\end{align}
This completes the proof of Lemma \ref{lem_en_est}.
\end{proof}
\begin{remark}
In the case $q=2$,
we have to take
$\alpha = 1$,
otherwise the term
\begin{align}
    (\alpha -1) \int_{D\times S^1} |\Omega(r,z)|^2
    r^{\alpha-1} v^r \eta_R \,drdz
\end{align}
remains in the identity \eqref{eq:energy:q=2}.
When $\alpha \neq 1$,
the above proof does not seem to work well.
\end{remark}

\subsection{Pointwise behavior via maximum principle}

From \eqref{Diri_o} and Lemma \ref{lem_en_est},
we have
\begin{align*}%
	\int_{-\pi}^{\pi} \int_{r_1}^{\infty} r^{q+1} |\Omega(r,z)|^q \,drdz
	< \infty,\quad
	\int_{-\pi}^{\pi} \int_{r_1}^{\infty} r^{\alpha} |\Omega(r,z)|^{q-2} |\nabla \Omega(r,z)|^2 \,drdz < \infty
\end{align*}%
with
\begin{align}
    \alpha = \begin{dcases}
    q + 1 - 2/q & (q>2),\\
     1 & (q=2).
    \end{dcases}
\end{align}

The following proposition shows that
the above bounds with the maximum principle yield a pointwise behavior of $\Omega$
as $r \to \infty$.

\begin{proposition}\label{prop_pt}
Let $r_1 > 0$,
and let
$f = f(r,z) \in C^1((r_1,\infty) \times S^1)$
be $2\pi$-periodic in $z$
and satisfy
\begin{align}%
\label{ass_f}
	\int_{-\pi}^{\pi} \int_{r_1}^{\infty} r^{\alpha} | f (r,z)|^{q-2} |\nabla_{r,z} f(r,z) |^2 \,drdz
	+ \int_{-\pi}^{\pi} \int_{r_1}^{\infty} r^{q+1} | f(r,z)|^q \,drdz < \infty
\end{align}%
with some
$q \in [2,\infty)$ and $\alpha \in \re$.
Moreover, we assume that $f$ satisfies the maximum principle in
$(r_1,\infty) \times S^1$,
that is,
for every 
$\rho_1, \rho_2 \in (r_1, \infty)$
with $\rho_1 < \rho_2$,
the function
$f$
restricted in
$[\rho_1, \rho_2] \times S^1$
attains its maximum at the boundary
$r = \rho_1$ or $r = \rho_2$.
Then, we have
\begin{align*}
	\lim_{r\to \infty} r^{\beta/q} \sup_{-\pi \le z \le \pi} |f(r,z)| = 0
\end{align*}
with
$\beta = \min\{ q + 2, \frac{\alpha + q +3}{2} \}$.
\end{proposition}
\begin{proof}
For $n \in \mathbb{N}$ satisfying
$2^n > r_1$,
by the assumption and the Schwarz inequality, we have
\begin{align}
    \infty &>
    \int_{2^n}^{2^{n+1}} \int_{-\pi}^{\pi}
    |f|^{q-2} \left( r^{q+1} |f|^2 +r^{\alpha} |\nabla_{r,z} f|^2 \right) \,dzdr \\
    &\ge
    \int_{2^n}^{2^{n+1}} \int_{-\pi}^{\pi}
    |f|^{q-2} \left( r^{q+1} |f|^2 + r^{\frac{\alpha+q+1}{2}} | f | |\nabla_{r,z} f| \right)\,dzdr \\
    &=
    \int_{2^n}^{2^{n+1}} \frac{dr}{r}
    \int_{-\pi}^{\pi} |f|^{q-2} \left(
        r^{q+2} |f|^2 + r^{\frac{\alpha+q+3}{2}} | f | |\nabla_{r,z} f|
    \right) \,dz.
\end{align}
By the mean value theorem for integration,
there exists a sequence
$r_n \in [2^n, 2^{n+1}]$
such that the right-hand side of the above equals to
\begin{align}
    \log 2 \int_{-\pi}^{\pi}
    |f(r_n,z)|^{q-2}
    \left( 
    r_n^{q+2} |f(r_n,z)|^2
    + r_n^{\frac{\alpha+q+3}{2}}
    |f(r_n,z)| | \nabla_{r,z} f(r_n,z)|
    \right) \,dz.
\end{align}
By the fundamental theorem of calculus, for $z_1, z_2 \in S^1$ 
we have with
$\beta = \min\{ q+2, \frac{\alpha+q+3}{2} \}$ that 
\begin{align}
    r_n^{\beta} |f(r_n,z_1)|^q - r_n^{\beta} |f(r_n,z_2)|^q
    &=
    r_n^{\beta} \left| \int_{z_2}^{z_1} \partial_z |f(r_n,z)|^q \,dz \right| \\
    &\le C r_n^{\beta}
    \int_{-\pi}^{\pi} |f(r_n,z)|^{q-1} |\nabla_{r,z} f(r_n,z)| \,dz.
\end{align}
Integrating both sides of the above inequality on
$[-\pi,\pi]$
with respect to $z_2$,
we deduce
\begin{align}
    r_n^{\beta} | f(r_n,z_1)|^q
    &\le C r_n^{\beta} \int_{-\pi}^{\pi} | f(r_n,z)|^q \,dz \\
    &\quad
    + C r_n^{\beta}
    \int_{-\pi}^{\pi} |f(r_n,z)|^{q-1} |\nabla_{r,z} f(r_n,z)| \,dz.
\end{align}
Noting
$\beta = \min\{ q + 2, \frac{\alpha + q +3}{2} \}$
and by the Lebesgue dominated convergence theorem,
we have
\begin{align}
    \lim_{n\to \infty} r_n^{\beta}
    \sup_{-\pi \le z \le \pi} | f(r_n,z)|^q = 0.
\end{align}

For
$r \in (r_1, \infty)$,
we take $n\in \mathbb{N}$ so that
$r \in [r_n, r_{n+1}]$
and apply the maximum principle to obtain
\begin{align}
    \sup_{[r_n, r_{n+1}] \times S^1} r^{\beta/q} |f(r,z)|
    &\le
    C \max \left\{
        \sup_{-\pi \le z \le \pi} r_n^{\beta/q} |f(r_n,z)|, \ 
        \sup_{-\pi \le z \le \pi} r_{n+1}^{\beta/q} |f(r_{n+1},z)|
    \right\} \\
    &\to 0\quad (n\to \infty).
\end{align}
This implies
\begin{align}
    \lim_{r\to \infty} r^{\beta/q}
    \sup_{-\pi \le z \le \pi} | f(r,z) | = 0,
\end{align}
and the proof is complete.
\end{proof}

\subsection{Proof of Theorem \ref{thm_asymp}}
By the assumptions on Theorem \ref{thm_asymp}
and the result of Lemma \ref{lem_en_est},
we can apply Proposition \ref{prop_pt}
with
\begin{align}
    \alpha = \begin{dcases}
    q + 1 - 2/q & (q>2),\\
     1 & (q=2).
    \end{dcases}
\end{align}
In this case, we have
\begin{align}
    \beta
    &=
    \min \left\{ q + 2, \frac{\alpha + q +3}{2} \right\}
    =
    \begin{dcases}
    q+2-\frac{1}{q} & (q>2),\\
    3 &(q=2).
    \end{dcases}
\end{align}
Therefore, we obtain
\begin{align}
    \begin{dcases}
    \lim_{r\to \infty} r^{ 1+ 2/q - 1/q^2}
	\sup_{z \in [-\pi, \pi]} |\Omega(r,z)| = 0
	&(q>2),\\
	\lim_{r\to \infty} r^{3/2}
	\sup_{z \in [-\pi, \pi]} |\Omega(r,z)| = 0
	&(q=2).
    \end{dcases}
\end{align}
Concerning the estimate of  $\omega^{\theta}$, we have by the relation 
$\omega^{\theta} = r \Omega$ that 
\begin{align}
    \begin{dcases}
    \lim_{r\to \infty} r^{2/q - 1/q^2}
	\sup_{z \in [-\pi, \pi]} |\omega^{\theta} (r,z)| = 0
	&(q>2),\\
	\lim_{r\to \infty} r^{1/2}
	\sup_{z \in [-\pi, \pi]} |\omega^{\theta} (r,z)| = 0
	&(q=2).
    \end{dcases}
\end{align}
This completes the proof of Theorem \ref{thm_asymp}.

\section{Proof of Theorem \ref{thm:liouville}}
We first remark that
the maximum principle for $\Omega$
cannot be directly applicable
due to the presence of $z$-axis.
Therefore, we repeat a similar argument of
the proof of Lemma \ref{lem_en_est}
with $r_0 = 0$.
We first note that, corresponding to the property 
\eqref{Diri_o}, we have
\begin{align}\label{eq:li:omega}
    \int_{-\pi}^{\pi} \int_0^{\infty} r^{q+1} |\Omega|^q \,drdz < \infty.
\end{align}
Let us first take $\alpha \in (1,3)$,
and let
$h(\Omega)$
be a nonnegative, $C^1$, and piecewise $C^2$ function
such that
$h''$ is locally bounded.
We use the cut-off function
$\xi_2$
in \eqref{eq:testfunc},
and for $R>0$ to define
\begin{align}
    \eta_R (r,z) = \xi_2 \left( \frac{r}{R} \right). 
\end{align}
In the same way to the proof of Lemma \ref{lem_en_est},
we start with the same identity as
\eqref{eq:div:form:2}:
\begin{align}
\label{eq:li:id}
	&\diver_{r,z} \left[ (r^{\alpha} \eta_R) \nabla_{r,z}h(\Omega)
					- h(\Omega) \nabla_{r,z} (r^{\alpha} \eta_R )
					- (r^{\alpha} \eta_R) h(\Omega) v \right] \\
	&\quad 
	+ 3 \partial_r (r^{\alpha-1} \eta_R h(\Omega)) \\
	&= r^{\alpha} \eta_R h''(\Omega) |\nabla_{r,z} \Omega|^2
		- h(\Omega) \left[ r^{\alpha} \partial_{r}^2 \eta_R + (2\alpha-3) r^{\alpha-1} \partial_r \eta_R \right] \\
	&\quad - h(\Omega) \left[ r^{\alpha} v \partial_{r} \eta_R + (\alpha-1) r^{\alpha-1} v^r \eta_R \right] \\
	&\quad
	-(\alpha-3)(\alpha-1) h(\Omega) r^{\alpha-2} \eta_R.
\end{align}
Now, we integrate both sides of the above identity over
$[0,R] \times S^1$.
Since $\alpha > 1$,
we have
\begin{align}
    \left. \left(
    r^{\alpha} \eta_{R} \nabla_{r, z} h(\Omega)
    -h(\Omega) \nabla_{r,z}
    \left(r^{\alpha} \eta_{R}\right)
    -\left(r^{\alpha} \eta_{R}\right) h\left(\Omega \right) v
    +3 r^{\alpha-1} \eta_R h(\Omega) 
    \right) \right|_{r=0} = 0.
\end{align}
From this and periodicity in $z$-direction,
we deduce
\begin{align}
	&\int_{-\pi}^{\pi} \int_0^R r^{\alpha}
	\eta_R h''(\Omega) |\nabla_{r,z} \Omega(r,z) |^2
	\,drdz \\
	&= \int_{-\pi}^{\pi} \int_0^R h(\Omega)
		\left[ r^{\alpha} \partial_r^2 \eta_R
		+ (2\alpha -3) r^{\alpha-1} \partial_r \eta_R
		\right] \, drdz \\
	&\quad + \int_{-\pi}^{\pi} \int_0^R h(\Omega)
		\left[
		r^{\alpha} v^r \partial_r \eta_R
		+ (\alpha-1) r^{\alpha-1} v^r \eta_R
		\right]\,drdz \\
	&\quad
	+ (\alpha-3)(\alpha-1)\int_{-\pi}^{\pi} \int_{0}^{R}
	 r^{\alpha-2}
	h(\Omega) \eta_{R} \,drdz.
\end{align}
Since $h(\Omega)$ is nonnegative
and $\alpha \in (1,3)$,
the last term can be dropped,
and we obtain
\begin{align}
	&\int_{-\pi}^{\pi} \int_0^R r^{\alpha}
	\eta_R h''(\Omega) |\nabla_{r,z} \Omega(r,z) |^2
	\,drdz \\
	&\le \int_{-\pi}^{\pi} \int_0^R h(\Omega)
		\left[ r^{\alpha} \partial_r^2 \eta_R
		+ (2\alpha -3) r^{\alpha-1} \partial_r \eta_R
		\right] \, drdz \\
	&\quad + \int_{-\pi}^{\pi} \int_0^R h(\Omega)
		\left[
		r^{\alpha} v^r \partial_r \eta_R
		+ (\alpha-1) r^{\alpha-1} v^r \eta_R
		\right]\,drdz.
\end{align}
Moreover, since
$h''$ is locally bounded,  
by letting
$\alpha \to 1+0$ in the above estimate, 
we have
\begin{align}
\label{eq:li:est2}
	&\int_{-\pi}^{\pi} \int_0^R r
	\eta_R h''(\Omega) |\nabla_{r,z} \Omega(r,z) |^2
	\,drdz \\
	&\le \int_{-\pi}^{\pi} \int_0^R h(\Omega)
		\left[ r \partial_r^2 \eta_R
		- \partial_r \eta_R
		\right] \, drdz \\
	&\quad + \int_{-\pi}^{\pi} \int_0^R h(\Omega)
		r v^r \partial_r \eta_R \,drdz.
\end{align}

Now, we estimate the right-hand side.
For the case $q > 2$,
we take
$h(\Omega) = |\Omega|^q$.
Then, by using
\eqref{est_eta} and Lemma \ref{lem:morrey},
the right-hand side of \eqref{eq:li:est2}
is estimated by
\begin{align*}
    &C \int_{-\pi}^{\pi} \int_{R/2}^R 
    |\Omega (r,z)|^q
    \left(
    r R^{-2} + R^{-1} + r (1+r)^{1-2/q} R^{-1}
    \right) \,dr dz.
\end{align*}
Using the condition \eqref{eq:li:omega},
we easily see that the above integral
tends to $0$ as $R\to \infty$.
Namely, letting
$R \to \infty$ in \eqref{eq:li:est2},
we have
\begin{align}
    \int_{-\pi}^{\pi} \int_0^{\infty}
	r |\Omega(r,z)|^{q-2} 
	|\nabla_{r,z} \Omega(r,z) |^2 \,drdz = 0.
\end{align}
This implies
$\nabla |\Omega|^{q/2} = 0$
and $\Omega$ is constant.
By the bound \eqref{eq:li:omega}, we have
$\Omega = 0$,
and so is
$\omega^{\theta}$.

Next, when
$q = 2$,
we take
\begin{align}
    h(\Omega) = 
    \begin{dcases}
    |\Omega|^2 &(|\Omega| \le A),\\
    A(2|\Omega| - A) &(|\Omega| \ge A),
    \end{dcases}
\end{align}
with $A>0$, and perform the completely same argument as in
the proof of Lemma \ref{lem_en_est}.
Hence, we can obtain
\begin{align}
    \int_{(0,\infty) \times S^1 \atop |\Omega| \le A}
    r \eta_R |\nabla_{r,z} \Omega|^2
    \,drdz
    &\le
    C A R^{-2}
    \left( \int_{-\pi}^{\pi} \int_{\frac{R}{2}}^R r^3 |\Omega|^2 \,drdz \right)^{1/2}
    \| \nabla v \|_{L^2} \\
    &\quad
    + C (\log R)^{1/2} R^{-3}
    \int_{-\pi}^{\pi} \int_{\frac{R}{2}}^{R}
    r^3 |\Omega|^2 \,drdz.
\end{align}
Letting first $R\to \infty$, and then $A \to \infty$,
we have
\begin{align}
    \int_{-\pi}^{\pi} \int_0^{\infty}
	r |\nabla_{r,z} \Omega(r,z) |^2 \,drdz = 0,
\end{align}
which also implies
$\nabla \Omega = 0$.
Thus, in the same way as before,
we see that
$\omega^{\theta} = 0$.

Finally, the formula
\begin{align}
    0 = \nabla \times (\nabla \times v)
    = \nabla (\nabla \cdot v) - \Delta v
\end{align}
and $\nabla \cdot v = 0$
lead to $\Delta v = 0$.
Therefore, the assumption
$\nabla v \in L^q (\re^2 \times S^1)$
and the periodicity of $v$ with respect to $z$
show that
$v$
must be constant.
This completes the proof.
\qed

\section*{Acknowledgement}
This work was supported by JSPS 	
Grant-in-Aid for Scientific Research (A)
Grant Number JP21H04433.

\end{document}